\documentclass[11pt]{amsart}

\usepackage{amsmath,amssymb,amsthm}
\usepackage[margin=1.15in]{geometry}

\newtheorem{theorem}{Theorem}[section]
\newtheorem{proposition}[theorem]{Proposition}
\newtheorem{lemma}[theorem]{Lemma}
\newtheorem{corollary}[theorem]{Corollary}
\newtheorem{remark}[theorem]{Remark}

\newcommand{\R}{\mathbb R}
\newcommand{\HH}{\mathbf H}

\newcommand{\inner}[2]{\left\langle #1,#2\right\rangle}

\title{Rigidity of Ricci-Pinched Complete Self-Shrinkers in Arbitrary Codimension}
\author{Juan Li}
\address{School of Mathematics, Hangzhou Normal University, Hangzhou, China}
\email{juanli@zju.edu.cn}

\author{Zhiyuan Xu}
\address{School of Mathematics, Hangzhou Normal University, Hangzhou, China}
\email{xuzy@hznu.edu.cn}
\thanks{Zhiyuan Xu is the corresponding author.}


\begin{document}

\begin{abstract}
Let $M$ be an $n$-dimensional complete connected self-shrinker in $\R^{n+p}$. Denote by $\operatorname{Ric}$ and $\HH$ the Ricci curvature tensor and the mean curvature vector of $M$, respectively. We prove that if the self-shrinker $M$ satisfies $\operatorname{Ric}\ge(\frac{n-2}{n^2}+\varepsilon_n)|\HH|^2g$, then $M$ is either a linear subspace or a round shrinking sphere, where $\varepsilon_n$ is an explicit positive constant equal to $\frac{1}{3n^2}+O\left(\frac{1}{n^5}\right)$. Moreover, we obtain a rigidity theorem that is sharp in codimension two for the self-shrinker satisfying $\operatorname{Ric}\ge\frac{n-2}{n^2}|\HH|^2g$.
\end{abstract}

\maketitle

\section{Introduction}

The mean curvature flow is an evolution equation where a submanifold evolves over time by locally moving in the direction of steepest descent for the volume element. An immersed submanifold $x:M^n\to\R^{n+p}$ is called a self-shrinker if it satisfies the equation
$$
  \HH=-\lambda x^\perp,\qquad \lambda>0 .
$$
This is equivalent, under a suitable spatial rescaling, to the standard forms $\HH=-x^\perp$ and $\HH=-\frac{1}{2}x^\perp$. Self-shrinkers form one of the basic models in the study of the mean curvature flow, since they arise as homothetically shrinking solutions and as tangent flows at singularities. The simplest examples of self-shrinkers are round spheres, planes and circular cylinders. A central rigidity question concerns which geometric assumptions force a self-shrinker of higher codimension to reduce to one of lower codimension.

In the case of codimension one, classification results under convexity hypotheses are well established.  For curves, the closed homothetic shrinkers were classified by Abresch--Langer \cite{AbreschLanger1986}.  In higher dimensions, Huisken \cite{Huisken1990} proved that compact mean-convex hypersurface self-shrinkers are round spheres. Furthermore, for complete self-shrinkers with polynomial volume growth and bounded second fundamental form, he \cite{Huisken1993} classified the generalized cylinders. Colding--Minicozzi \cite{ColdingMinicozzi2012} later removed the boundedness assumption on the second fundamental form for complete mean-convex hypersurface self-shrinkers. In addition, Brendle \cite{Brendle2016} proved that the round sphere is the only compact embedded self-shrinker of genus zero in $\R^3$. 

In higher codimension, the situation is more delicate. There is no canonical choice of orientation to define a signed scalar mean curvature, and the condition $|\HH|>0$ alone does not force codimension reduction: minimal submanifolds of self-shrinking spheres already give many genuine self-shrinkers of higher codimension. Consequently, additional structure is needed. Smoczyk \cite{Smoczyk2005} showed that the parallel principal normal condition leads to a precise classification in arbitrary codimension; this direction was subsequently refined by Li--Wei \cite{LiWei2014}. Andrews--Li--Wei \cite{AndrewsLiWei2014} extended Colding--Minicozzi's F-stability theory to arbitrary codimension, classifying F-stable self-shrinkers as cylinders with nonnegative mean curvature. Arezzo--Sun \cite{ArezzoSun2013} further showed that under the additional assumptions of flat normal bundle and $\HH\neq0$, the only such self-shrinker is the round sphere.

A second approach is to impose pinching conditions on the second fundamental form. In codimension one, Le--Sesum \cite{LeSesum2010} proved a gap theorem which, under the normalization $\HH=-x^\perp$, forces a complete embedded hypersurface self-shrinker, with polynomial volume growth and $|A|^2<1$, to be a hyperplane. Later, Cao--Li \cite{CaoLi2013} extended this gap phenomenon to arbitrary codimension, proving the optimal classification under $|A|^2\le1$; Cheng--Peng \cite{ChengPeng2012} subsequently removed the polynomial volume growth assumption by applying a generalized maximum principle to the drift operator. Rigidity results with $0\le|A|^2-1\le\delta(n)$ have been obtained by many authors \cite{ChengWei2015,DingXin2014,LeiXuXu2020,XuXu2017,Zhao2025}. Moreover, Ding--Xin--Yang \cite{DingXinYang2016} proved that complete self-shrinkers with Gauss image in a hemisphere are hyperplanes or cylinders in codimension one, and self-shrinking graphs with slope less than 3 are linear subspaces in higher codimension. Cao--Xu--Zhao \cite{CaoXuZhao2024} and Ding--Ge--Li \cite{DingGeLi2026} studied pinching theorems in higher codimension under the assumption of bounded trace-free second fundamental form. In 2022, Naff \cite{Naff2022} proved that closed flows in higher codimension satisfying $|A|^2\le c_n|\HH|^2$ become asymptotically codimension one near singularities. Subsequently, Lee--Naff--Zhu \cite{LeeNaffZhu2026} established a scale-invariant version of this planarity estimate, together with a convexity estimate for pinched ancient solutions. As a consequence, they classified uniformly $c_n$-pinching shrinkers of bounded entropy as generalized cylinders. More recently, Impera--Rimoldi--Ruatta \cite{ImperaRimoldiRuatta2026} used a purely elliptic argument based on weighted parabolicity in Gaussian space to classify properly immersed $\frac{4}{3n}$-pinching self-shrinkers of arbitrary codimension as planes or generalized cylinders.

In contrast to these pinching approaches, our aim is to investigate the rigidity of complete self-shrinkers in arbitrary codimension under the condition $\operatorname{Ric} \ge \frac{n-2}{n^2}|\HH|^2g$. In what follows, at points where $|\HH|>0$, we use a normal frame adapted to the mean curvature vector, with its first normal direction along $\HH$.  We use
$$
  |A^\perp|^2=\sum_{\alpha\ge n+2}|A^\alpha|^2,\qquad\rho^\perp=\sum_{\alpha,\beta=n+1}^{n+p}|[A^\alpha,A^\beta]|^2 .
$$
Thus $A^\perp$ is the component of the second fundamental form orthogonal to the mean curvature direction, and $\rho^\perp$ is the unnormalized squared normal curvature. We first establish the following rigidity theorem under an almost sharp lower bound on the Ricci curvature.

\begin{theorem} \label{thm:ricci-improved}
Let $x:M^n\to\R^{n+p}$ be a complete connected immersed self-shrinker with $n\ge2$. Suppose that
$$
  \operatorname{Ric}\ge\mu_n|\HH|^2g
$$
holds on $M$.  Then, after an orthogonal transformation of $\R^{n+p}$, the image $x(M)$ is either the linear subspace $\R^n$ through the origin or the round self-shrinking sphere $\mathbb S^n\left(\sqrt{\frac{n}{\lambda}}\right)$ contained in an $(n+1)$-dimensional linear subspace.
\end{theorem}
Here $\mu_n$ is an explicit positive constant (see \eqref{eq:mu-epsilon-exact} in Section \ref{sec:ricci-pinching}), and it can be written as
$$
  \mu_n=\begin{cases}
    \frac{1}{6}, & n=2,\\
    \frac{n-2}{n^2}+\varepsilon_n, & n\ge3,
  \end{cases}
$$
where
$$
  \varepsilon_n=\frac{1}{3n^2}+O\left(\frac{1}{n^5}\right)\quad\text{as }n\to\infty.
$$

Moreover, under the sharp lower bound $\operatorname{Ric}\ge\frac{n-2}{n^2}|\HH|^2g$, we establish a rigidity theorem.  No additional hypothesis is required in codimension one or two, whereas in higher codimension we impose the normal curvature condition stated below.

\begin{theorem} \label{thm:critical-main}
Let $x:M^n\to\R^{n+p}$ be a complete connected immersed self-shrinker with $n\ge4$. Assume that
$$
  \operatorname{Ric} \ge \frac{n-2}{n^2}|\HH|^2g
$$
holds on $M$. If $p=1,\,2$, or $p\ge3$ and 
$$
  \rho^\perp \le \left( R-\frac{n-2}{n}|\HH|^2 \right)|A^\perp|^2,
$$
then, after an orthogonal transformation of $\R^{n+p}$, the image $x(M)$ is one of the following:
\begin{enumerate}
\item[\textup{(i)}] The linear subspace $\R^n$ through the origin;
\item[\textup{(ii)}] The round self-shrinking sphere $\mathbb S^n\left(\sqrt{\frac{n}{\lambda}}\right)$ contained in an $(n+1)$-dimensional linear subspace;
\item[\textup{(iii)}] The Euclidean dilation of the Clifford hypersurface $\mathbb S^{\frac{n}{2}}
  \left(\sqrt{\frac{n}{2\lambda}}\right)\times\mathbb S^{\frac{n}{2}}\left(\sqrt{\frac{n}{2\lambda}}\right)$ with $p=2$ and $n$ even.
\end{enumerate}
\end{theorem}

\begin{remark}\label{rem:four-dimensional-critical}
Let $\widetilde\Sigma^n\subset\mathbb S^{n+1}(1)$ be a minimal hypersurface.  Then the dilation $\Sigma=\sqrt{\frac{n}{\lambda}}\,\widetilde\Sigma\subset\R^{n+2}$ satisfies $\HH_\Sigma=-\lambda x^\perp$.  When $n$ is even, the nonspherical equality model
\[
  \mathbb S^{\frac{n}{2}}\left(\sqrt{\frac{n}{2\lambda}}\right)\times\mathbb S^{\frac{n}{2}}\left(\sqrt{\frac{n}{2\lambda}}\right)
\]
is obtained from the minimal Clifford hypersurface
\[
  T\left(\frac{\pi}{4}\right)= \mathbb S^{\frac{n}{2}}\left(\frac{1}{\sqrt2}\right)\times\mathbb S^{\frac{n}{2}}\left(\frac{1}{\sqrt2}\right).
\]
Liu--Terng~\cite[Theorem~1.1 and Example~4.4]{LiuTerng2020} showed that the spherical mean curvature flow within the corresponding Clifford family is ancient and converges to $T\left(\frac{\pi}{4}\right)$ as $t\to-\infty$. Among the members of this family, only the minimal Clifford hypersurface generates a homothetically shrinking Euclidean mean curvature flow. Moreover, its Euclidean dilation satisfies
\[
  \operatorname{Ric}=\frac{n-2}{n^2}|\HH|^2g,
\]
and hence attains equality in the Ricci lower bound.

The Ricci bound alone is nevertheless satisfied by the standard Veronese minimal embedding $\mathbb{CP}^2(4/3)\subset\mathbb S^7(1)$. After dilation by $\sqrt{\frac{4}{\lambda}}$, it yields a $4$-dimensional self-shrinker in $\R^8$ satisfying
$$
\begin{gathered}
  |\HH|^2=4\lambda,\qquad
  \operatorname{Ric}=\frac{\lambda}{2}g=\frac{1}{8}|\HH|^2g,
  \qquad R=\frac{1}{2}|\HH|^2,\\
  \rho^\perp=\frac{1}{24}|\HH|^4=\frac{1}{6}R^2
  =\frac{2}{3}|A^\perp|^4>0.
\end{gathered}
$$
Thus this model realizes equality in the Ricci bound but is excluded by the additional normal curvature condition.

Note that Li--Wei \cite{LiWei2014} classified closed Ricci pinched self-shrinkers in higher codimension under additional spherical reduction hypotheses. Their argument does not extend to the complete case: the integration by parts step requires closedness (or at least polynomial volume growth), and the spherical minimal submanifold classifications by Ejiri and Li are inherently compactness results. We assume neither closedness, polynomial volume growth, parallel principal normal, nor a priori constancy of the extrinsic radius. Instead, the Ricci lower bound first yields properness, then bounds the eigenvalues of $\mathring A^{n+1}/|\HH|$ and estimates the zeroth-order terms in the quotient identity. The quotient identity developed below ensures the required spherical reduction; at the critical coefficient, this holds without auxiliary conditions in Euclidean codimension one or two, and under the stated normal curvature condition in higher codimension. Tracing the Ricci lower bound yields the $\frac{2}{n}$-pinching condition, which is weaker than that of Impera--Rimoldi--Ruatta \cite{ImperaRimoldiRuatta2026}. 

\end{remark}

When $\operatorname{Ric}\ge\frac{n-2}{n^2}|\HH|^2g$, the flat normal bundle implies the additional normal curvature bound. This yields the following corollary.

\begin{corollary}
    Let $x:M^n\to\R^{n+p}$ be a complete connected immersed self-shrinker with flat normal bundle.  If $n\ge4$ and
$$
  \operatorname{Ric}\ge\frac{n-2}{n^2}|\HH|^2g
$$
holds on $M$, then, after an orthogonal transformation of $\R^{n+p}$, the image $x(M)$ is either $\R^n$, $\mathbb S^n\left(\sqrt{\frac{n}{\lambda}}\right)$ or $\mathbb S^{\frac{n}{2}}\left(\sqrt{\frac{n}{2\lambda}}\right)\times\mathbb S^{\frac{n}{2}}\left(\sqrt{\frac{n}{2\lambda}}\right)$.
\end{corollary}

The main analytic tool is a quotient identity for $|\mathring A|^2/|\HH|^2$, where $\mathring A$ denotes the trace-free second fundamental form. Section~2 develops the analytic framework by combining the self-shrinker Simons identities for $|A|^2$ and $|\HH|^2$ to obtain a drift Laplacian formula for this quotient. The resulting operator is symmetric with respect to the weighted measure $|\HH|^4d\mu_f$, and a weighted Liouville argument converts the nonnegativity of $\left( \mathcal L_\lambda+2\inner{\nabla\log |\HH|^2}{\nabla\cdot} \right)\frac{|\mathring A|^2}{|\HH|^2}$ into rigidity. The positive Ricci lower bounds in the main theorems first imply that the immersion is proper and hence has polynomial volume growth, so no volume growth assumption is required. We then prove that the Ricci lower bound forces either the immersion to be a linear subspace or $|\HH|>0$ everywhere. When $|\HH|>0$ and $\nabla^\perp\HH\equiv0$, a spherical reduction argument shows that the self-shrinker is a compact minimal submanifold of the shrinker sphere. For the almost sharp Ricci lower bound, Section~3 derives the required zeroth-order estimate for $n\ge3$ by jointly optimizing the norm and spectral diameter of $\mathring A^{n+1}$. It then determines the corresponding constant $\mu_n$. The same section establishes the algebraic estimate needed at the sharp Ricci lower bound in codimension two. Section~4 combines these auxiliary results to prove the main theorems.

\section{Preliminaries and auxiliary results}
Let $x:M^n\to\mathbb R^{n+p}$ be an $n$-dimensional immersion.  We use the following convention for indices:
  $$
    1\le i,j,k,\ldots\le n,\qquad n+1\le \alpha,\beta,\gamma,\ldots\le n+p,\qquad 1\le A,B,C,\ldots\le n+p .
  $$
Let $\{e_A\}$ be a local orthonormal frame of $\mathbb R^{n+p}$ along $M$, chosen so that $e_i$ are tangent to $M$ and $e_\alpha$ are normal to $M$.  Let $\{\theta_A\}$ and $\{\theta_{AB}\}$ denote the associated dual coframe and connection forms.  The second fundamental form is written as
  $$
  A=\sum_{\alpha=n+1}^{n+p}A^\alpha e_\alpha,\qquad A^\alpha=\sum_{i,j=1}^n A_{ij}^\alpha\,\theta_i\otimes\theta_j,
  $$
where $A_{ij}^\alpha=A_{ji}^\alpha$.  The mean curvature vector $\HH$ is the trace of the second fundamental form defined by
  $$
  \HH=\operatorname{tr}A=\sum_{\alpha=n+1}^{n+p}H^\alpha e_\alpha,\qquad H^\alpha=\sum_{i=1}^n A_{ii}^\alpha .
  $$
Hence
  $$
  |A|^2=\sum_{\alpha=n+1}^{n+p}|A^\alpha|^2=\sum_{\alpha=n+1}^{n+p}\sum_{i,j=1}^n(A_{ij}^\alpha)^2,\qquad  |\HH|^2=\sum_{\alpha=n+1}^{n+p}(H^\alpha)^2=\sum_{\alpha=n+1}^{n+p} \left(\sum_{i=1}^n A_{ii}^\alpha\right)^2 .
  $$
The trace-free second fundamental form is defined by
  $$
    \mathring A^\alpha
    =
    A^\alpha-\frac{H^\alpha}{n}\operatorname{Id},
    \qquad
    |\mathring A|^2
    =
    |A|^2-\frac{1}{n}|\HH|^2,
  $$
where $\operatorname{Id}$ denotes the $n\times n$ identity matrix with respect to the chosen orthonormal tangent frame. Throughout the paper, we denote the full ordered commutator term by
  $$
  \rho^\perp := \sum_{\alpha,\beta=n+1}^{n+p}|[A^\alpha,A^\beta]|^2 .
  $$
The Gauss equations are
$$
\begin{aligned}
R_{i j k l}=&\sum_\alpha\left(A_{i k}^\alpha A_{j l}^\alpha-A_{i l}^\alpha A_{j k}^\alpha\right), \\
R_{i k}=&\sum_\alpha H^\alpha A_{i k}^\alpha-\sum_{\alpha, j} A_{i j}^\alpha A_{j k}^\alpha, \\
R=&|\HH|^2-|A|^2,
\end{aligned}
$$
where $R=\sum_{i,j}R_{ijij}$ is the scalar curvature of $M$.

The Codazzi equations are given by
$$
A_{i j k}^\alpha=A_{i k j}^\alpha,
$$
where the covariant derivative of $A_{i j}^\alpha$ is defined by
$$
\sum_k A_{i j k}^\alpha \theta_k=d A_{i j}^\alpha
+\sum_k A_{k j}^\alpha \theta_{k i}+\sum_k A_{i k}^\alpha \theta_{k j}
+\sum_\beta A_{i j}^\beta \theta_{\beta \alpha}.
$$
Assume now that $M$ is a self-shrinker satisfying
  $$
    \HH=-\lambda x^\perp,\qquad \lambda>0 .
  $$
Here $x^\perp$ and $x^T$ denote the normal and tangential projections of the position vector, respectively.  We use the weighted Laplacian
  $$
    \mathcal L_\lambda=\Delta-\lambda\inner{x^T}{\nabla\cdot},
    \qquad d\mu_f=e^{-\frac{\lambda|x|^2}{2}}d\mu .
  $$
For self-shrinkers, the following Simons identities hold \cite{CaoLi2013}
$$
  \frac{1}{2}\mathcal L_\lambda |A|^2=|\nabla A|^2+\lambda |A|^2
  -\sum_{\alpha,\beta}\left(\sum_{i,j}A_{ij}^{\alpha}A_{ij}^{\beta}\right)^2-\rho^\perp,
$$
and
$$
  \frac{1}{2}\mathcal L_\lambda |\HH|^2=|\nabla^\perp\HH|^2+\lambda|\HH|^2
  -\sum_{\alpha,\beta}\left(\sum_{i,j}A_{ij}^{\alpha}A_{ij}^{\beta}\right)H^{\alpha}H^{\beta}.
$$
Subtracting $\frac{1}{n}$ times the identity for $|\HH|^2$ from the identity for $|A|^2$ therefore yields
$$
  \frac{1}{2}\mathcal L_\lambda|\mathring A|^2
  =|\nabla\mathring A|^2+\lambda|\mathring A|^2
  -\sum_{\alpha,\beta}\left(\sum_{i,j}A_{ij}^{\alpha}A_{ij}^{\beta}\right)^2-\rho^\perp
  +\frac{1}{n}\sum_{\alpha,\beta}\left(\sum_{i,j}A_{ij}^{\alpha}A_{ij}^{\beta}\right)H^{\alpha}H^{\beta}.
$$
At every point where $|\HH|>0$, the preceding Simons identities give the quotient identity
\begin{equation}\label{eq:quotient}
\begin{aligned}
  &\frac{1}{2}\left(
  \mathcal L_\lambda+2\inner{\nabla\log |\HH|^2}{\nabla\cdot}
  \right)\frac{|\mathring A|^2}{|\HH|^2}\\
  =&\frac{1}{|\HH|^2}\Bigg[
  |\nabla\mathring A|^2
  -\frac{|\mathring A|^2}{|\HH|^2}|\nabla^\perp\HH|^2-\rho^\perp\\
  &-\sum_{\alpha,\beta}
  \left(\sum_{i,j}A_{ij}^{\alpha}A_{ij}^{\beta}\right)^2
  +\frac{|A|^2}{|\HH|^2}\sum_{\alpha,\beta}
  \left(\sum_{i,j}A_{ij}^{\alpha}A_{ij}^{\beta}\right)
  H^{\alpha}H^{\beta}\Bigg].
\end{aligned}
\end{equation}

Choose an orthonormal normal frame such that $e_{n+1}=\HH/|\HH|$.  Then
$$
\begin{aligned}
  H^{n+1}=|\HH|&,\qquad
  H^\alpha=0\quad(\alpha\ge n+2),\\
  A^{n+1}=\mathring A^{n+1}+\frac{|\HH|}{n}\operatorname{Id}&,\qquad
  \sum_i A_{ii}^{\alpha}=0\quad(\alpha\ge n+2).
\end{aligned}
$$
We write $|A^\perp|^2:=\sum_{\alpha\ge n+2}|A^\alpha|^2$ for the component of the second fundamental form orthogonal to the mean curvature direction.  Since $A^\alpha=\mathring A^\alpha$ for $\alpha\ge n+2$, it follows that
$$
  |A^\perp|^2=|\mathring A|^2-|\mathring A^{n+1}|^2.
$$

Using the preceding frame identities, we have
$$
  \frac{|A|^2}{|\HH|^2}
  \sum_{\alpha,\beta}\left(\sum_{i,j}A_{ij}^{\alpha}A_{ij}^{\beta}\right)H^\alpha H^\beta
  =\left(|\mathring A^{n+1}|^2+|A^\perp|^2+\frac{|\HH|^2}{n}\right)
  \left(|\mathring A^{n+1}|^2+\frac{|\HH|^2}{n}\right).
$$
Separating the mean curvature direction from the remaining normal directions yields
\begin{equation}\label{hd-eq:algebra}
\begin{aligned}
  &\sum_{\alpha,\beta}
  \left(\sum_{i,j}A_{ij}^{\alpha}A_{ij}^{\beta}\right)^2
  +
  \rho^\perp -
  \frac{|A|^2}{|\HH|^2}
  \sum_{\alpha,\beta}
  \left(\sum_{i,j}A_{ij}^{\alpha}A_{ij}^{\beta}\right)
  H^{\alpha}H^{\beta}  \\
  =&
  -
  \left(
    |\mathring A^{n+1}|^2+\frac{|\HH|^2}{n}
  \right)
  |A^\perp|^2
  +
  2\sum_{\alpha\ge n+2}
  \left(\sum_{i,j}\mathring A_{ij}^{n+1}A_{ij}^{\alpha}\right)^2
  +
  \sum_{\alpha,\beta\ge n+2}
  \left(\sum_{i,j}A_{ij}^{\alpha}A_{ij}^{\beta}\right)^2  \\
  &+
  2\sum_{\alpha\ge n+2}|[\mathring A^{n+1},A^{\alpha}]|^2
  +
  \sum_{\alpha,\beta\ge n+2}|[A^{\alpha},A^{\beta}]|^2 .
\end{aligned}
\end{equation}

\begin{lemma}\label{hd-lem:normal-scalar-algebra}
At a point where $|\HH|>0$, in the adapted normal frame introduced above, one has
$$
  \sum_{\alpha,\beta}\left(\sum_{i,j}A_{ij}^{\alpha}A_{ij}^{\beta}\right)^2+\rho^\perp
  -
  \frac{|A|^2}{|\HH|^2}
  \sum_{\alpha,\beta}\left(\sum_{i,j}A_{ij}^{\alpha}A_{ij}^{\beta}\right)H^{\alpha}H^{\beta}
  \le\left(|\mathring A|^2-\frac{|\HH|^2}{n}\right)|A^\perp|^2+\rho^\perp .
$$
\end{lemma}

\begin{proof}
By the Cauchy--Schwarz inequality and the Gram matrix estimate,
$$
\begin{aligned}
  2\sum_{\alpha\ge n+2}
  \left(\sum_{i,j}\mathring A_{ij}^{n+1}A_{ij}^{\alpha}\right)^2
  &\le
  2|\mathring A^{n+1}|^2|A^\perp|^2,\\
  \sum_{\alpha,\beta\ge n+2}
  \left(\sum_{i,j}A_{ij}^{\alpha}A_{ij}^{\beta}\right)^2
  &\le
  |A^\perp|^4 .
\end{aligned}
$$
The commutator terms in \eqref{hd-eq:algebra} sum to $\rho^\perp$.  Consequently, we obtain
$$
\begin{aligned}
  &\sum_{\alpha,\beta}
  \left(\sum_{i,j}A_{ij}^{\alpha}A_{ij}^{\beta}\right)^2
  +
  \rho^\perp -
  \frac{|A|^2}{|\HH|^2}
  \sum_{\alpha,\beta}
  \left(\sum_{i,j}A_{ij}^{\alpha}A_{ij}^{\beta}\right)
  H^{\alpha}H^{\beta}  \\
  \le&
  -
  \left(
    \frac{|\HH|^2}{n}
    -|\mathring A^{n+1}|^2
    -|A^\perp|^2
  \right)
  |A^\perp|^2+\rho^\perp .
\end{aligned}
$$
Since $ |\mathring A|^2=|\mathring A^{n+1}|^2+|A^\perp|^2 $, the desired inequality follows.
\end{proof}

\begin{lemma}[Andrews--Baker \cite{AndrewsBaker2010}]\label{lem:gradient-estimate}
For a submanifold of Euclidean space, one has
$$
  |\nabla\mathring A|^2 \ge \frac{2(n-1)}{n(n+2)} |\nabla^\perp\HH|^2 .
$$
\end{lemma}

We shall use the following weighted Liouville lemma to complete the proofs of the main theorems.

\begin{lemma}\label{lem:weighted-liouville}
Let $x:M^n\to\R^{n+p}$ be a complete connected self-shrinker immersion with polynomial volume growth and $|\HH|>0$.  If a smooth function $\psi$ satisfies $0\le\psi\le C_0$ and
$$
  \left( \mathcal L_\lambda+2\inner{\nabla\log |\HH|^2}{\nabla\cdot} \right)\psi\ge0,
$$
then $\psi$ is constant on $M$.
\end{lemma}

\begin{proof}
Since $|\HH|\le \lambda |x|$, the polynomial volume growth assumption implies
  $$
  \begin{aligned}
    \int_M |\HH|^4\,d\mu_f
    &=
    \int_M |\HH|^4 e^{-\frac{\lambda |x|^2}{2}}\,d\mu  \\
    &\le
    \lambda^4\int_M |x|^4 e^{-\frac{\lambda |x|^2}{2}}\,d\mu  \\
    &\le
    \lambda^4
    \sum_{k=0}^{\infty}
    (k+1)^4 e^{-\frac{\lambda k^2}{2}}
    \mu\left(x^{-1}\left(B_{k+1}^{\mathbb R^{n+p}}(0)\right)\right)
    <\infty,
  \end{aligned}
  $$
where $B_r^{\mathbb R^{n+p}}(0)$ denotes the Euclidean ball of radius $r$. Fix $o\in M$, and let $\eta_R\in C_c^\infty(M)$ be a cutoff function satisfying
  $$
    0\le \eta_R\le 1,\qquad
    \eta_R\equiv1\ \text{on }B_R^M(o),\qquad
    \operatorname{supp}\eta_R\subset B_{2R}^M(o),\qquad
    |\nabla\eta_R|\le \frac{C}{R}.
  $$

For every compactly supported smooth function $\phi$, integration by parts with respect to the measure $|\HH|^4d\mu_f$ gives
  $$
  \begin{aligned}
    \int_M
    \phi
    \left(
      \mathcal L_\lambda
      +
      2\inner{\nabla\log|\HH|^2}{\nabla\cdot}
    \right)\psi
    |\HH|^4\,d\mu_f =
    -\int_M\inner{\nabla\phi}{\nabla\psi}
    |\HH|^4\,d\mu_f .
  \end{aligned}
  $$
Taking $\phi=\eta_R^2\psi$, and using $\psi\ge0$, we obtain
  $$
  \begin{aligned}
  0
  &\le
  \int_M
  \eta_R^2\psi
  \left(
    \mathcal L_\lambda
    +
    2\inner{\nabla\log|\HH|^2}{\nabla\cdot}
  \right)\psi
  |\HH|^4\,d\mu_f  \\
  &=
  -\int_M\inner{\nabla(\eta_R^2\psi)}{\nabla\psi}
  |\HH|^4\,d\mu_f  \\
  &=
  -\int_M\eta_R^2|\nabla\psi|^2|\HH|^4\,d\mu_f
  -2\int_M\eta_R\psi
  \inner{\nabla\eta_R}{\nabla\psi}
  |\HH|^4\,d\mu_f .
  \end{aligned}
  $$
By Young's inequality,
  \begin{equation*}
    \begin{aligned}
      \int_M\eta_R^2|\nabla\psi|^2|\HH|^4\,d\mu_f
    &\le
    2\int_M\eta_R|\psi|\,|\nabla\eta_R|\,|\nabla\psi|
    |\HH|^4\,d\mu_f \\
    &\le\frac{1}{2}
    \int_M\eta_R^2|\nabla\psi|^2|\HH|^4\,d\mu_f +
    2\int_M\psi^2|\nabla\eta_R|^2|\HH|^4\,d\mu_f .
    \end{aligned}
  \end{equation*}
Hence
  $$
  \begin{aligned}
    \int_M\eta_R^2|\nabla\psi|^2|\HH|^4\,d\mu_f
    \le
    4\int_M\psi^2|\nabla\eta_R|^2|\HH|^4\,d\mu_f  \le
    \frac{4C^2C_0^2}{R^2}
    \int_M|\HH|^4\,d\mu_f .
  \end{aligned}
  $$
Letting $R\to\infty$, we obtain
  $$
    \int_M|\nabla\psi|^2|\HH|^4\,d\mu_f=0.
  $$
Since $|\HH|>0$, we have $\nabla\psi\equiv0$, so $\psi$ is constant on $M$.

\end{proof}

\begin{remark}
The proof is based on the standard cutoff argument of Yau \cite{Yau1976} and the related weighted versions due to Charalambous--Lu \cite{CharalambousLu2014}. Cheng--Peng \cite[Theorem 3.1]{ChengPeng2012} established an Omori--Yau type maximum principle for the self-shrinker drift operator $\mathcal L_\lambda$ and used it to remove the polynomial volume growth assumption from the gap theorem of Cao--Li \cite{CaoLi2013}.  Their principle does not apply directly here, since the quotient inequality involves the modified operator
$$
  \mathcal L_\lambda
  +
  2\inner{\nabla\log|\HH|^2}{\nabla\mathord\cdot}.
$$
Indeed, the present hypotheses provide no uniform control of $\nabla\log|\HH|^2$, so the additional drift term cannot be controlled along an Omori--Yau sequence for $\mathcal L_\lambda$.  We therefore use the weighted cutoff argument in Lemma \ref{lem:weighted-liouville}, for which polynomial volume growth supplies the required integrability.
\end{remark}

We prove the following properness lemma under a Ricci lower bound, in the spirit of Chodosh--Li--Minter--Stryker \cite{ChodoshLiMinterStryker2026}, by combining the second variation formula along minimizing geodesics with the self-shrinker equation.
\begin{lemma}\label{lem:ricci-properness}
Let $x:M^n\to\R^{n+p}$ be a complete connected self-shrinker immersion.  If $n\ge2$ and there exists a positive constant $\mu$ such that
$$
  \operatorname{Ric}\ge\mu|\HH|^2g
$$
on $M$, then the immersion $x$ is proper.
\end{lemma}

\begin{proof}
If $M$ is compact, the conclusion is immediate.  We therefore assume that $M$ is noncompact.  For every unit tangent vector $v$, the Gauss equation and the Ricci lower bound give
$$
  \inner{\HH}{A(v,v)}=\operatorname{Ric}(v,v)
  +\sum_i|A(v,e_i)|^2\ge\mu|\HH|^2.
$$
Thus $\inner{\HH}{A(\,\cdot\,,\,\cdot\,)}$ is a nonnegative symmetric bilinear form whose trace equals $|\HH|^2$, and hence
\begin{equation}\label{eq:properness-pointwise}
  0\le\inner{\HH}{A(v,v)}\le|\HH|^2
  \le\frac{1}{\mu}\operatorname{Ric}(v,v).
\end{equation}

Fix $o\in M$. For each $q\in M$ with $\ell:=d_M(o,q)>2$, let $\gamma:[0,\ell]\to M$ be a minimizing geodesic from $o$ to $q$, parametrized by arc length. Along $\gamma$, extend $E_1=\dot\gamma$ to a parallel orthonormal frame
$E_1,E_2,\ldots,E_n$, and set
$$
  \varphi(s)
  =
  \begin{cases}
    s, & 0\le s\le1,\\
    1, & 1\le s\le\ell-1,\\
    \ell-s, & \ell-1\le s\le\ell.
  \end{cases}
$$

Let $I$ denote the index form of $\gamma$. Since $\gamma$ is minimizing, its index form is nonnegative. Then for each $a=2,\ldots,n$, the second variation formula gives
$$
\begin{aligned}
  0
  &\le
  I(\varphi E_a,\varphi E_a)\\
  &=
  \int_0^\ell
  \left(
    |\nabla_{\dot\gamma}(\varphi E_a)|^2
    -
    \varphi^2
    \inner{R(E_a,\dot\gamma)\dot\gamma}{E_a}
  \right)\,ds\\
  &=
  \int_0^\ell
  \left(
    |\varphi'|^2
    -
    \varphi^2
    \inner{R(E_a,\dot\gamma)\dot\gamma}{E_a}
  \right)\,ds,
\end{aligned}
$$
where the last equality follows from $\nabla_{\dot\gamma}E_a=0$.  Summing over $a=2,\ldots,n$ and using
$$
  \operatorname{Ric}(\dot\gamma,\dot\gamma)
  =\sum_{a=2}^n\inner{R(E_a,\dot\gamma)\dot\gamma}{E_a},
$$
we obtain
$$
  \int_0^\ell\varphi^2\operatorname{Ric}(\dot\gamma,\dot\gamma)\,ds
  \le(n-1)\int_0^\ell|\varphi'|^2\,ds.
$$
Moreover, $\operatorname{Ric}\ge0$ and $\varphi=1$ on $[1,\ell-1]$, while $|\varphi'|=1$ on $(0,1)$ and $(\ell-1,\ell)$, and $\varphi'=0$ elsewhere.  Therefore
$$
\begin{aligned}
  \int_1^{\ell-1}
  \operatorname{Ric}(\dot\gamma,\dot\gamma)\,ds
  &\le
  \int_0^\ell
  \varphi^2\operatorname{Ric}(\dot\gamma,\dot\gamma)\,ds\\
  &\le
  (n-1)\int_0^\ell|\varphi'|^2\,ds
  \\
  &=
  (n-1)\left(
    \int_0^1 1\,ds
    +
    \int_{\ell-1}^\ell 1\,ds
  \right)
  \\
  &=
  2(n-1).
\end{aligned}
$$
Since the unit tangent bundle of $\overline{B_1^M(o)}$ is compact by the Hopf--Rinow theorem and $\gamma([0,1])\subset\overline{B_1^M(o)}$, the preceding estimate gives the following bound for $1\le t\le\ell-1$
\begin{equation}\label{eq:properness-ricci-integral}
\begin{aligned}
  \int_0^t\operatorname{Ric}(\dot\gamma,\dot\gamma)\,ds
  &=
  \int_0^1\operatorname{Ric}(\dot\gamma,\dot\gamma)\,ds
  +
  \int_1^t\operatorname{Ric}(\dot\gamma,\dot\gamma)\,ds\\
  &\le
  \max_{\substack{
    y\in\overline{B_1^M(o)},\ v\in T_yM\\
    |v|=1
  }}
  \operatorname{Ric}_y(v,v)
  +2(n-1)
  =:C_0.
\end{aligned}
\end{equation}
Here $C_0$ is independent of $q$ and of the minimizing geodesic $\gamma$.

Along $\gamma$, the Gauss formula and the self-shrinker equation yield
$$
\begin{aligned}
  \frac{d^2}{ds^2}
  \frac{|x(\gamma(s))|^2}{2}
  &=
  |dx(\dot\gamma)|^2
  +
  \inner{x}{
    dx(\nabla_{\dot\gamma}\dot\gamma)
    +A(\dot\gamma,\dot\gamma)
  }\\
  &=
  1
  +
  \inner{x^\perp}{A(\dot\gamma,\dot\gamma)}\\
  &=
  1
  -
  \frac{1}{\lambda}
  \inner{\HH}{A(\dot\gamma,\dot\gamma)}.
\end{aligned}
$$
Integrating the preceding second derivative identity from $0$ to $t$ yields
\[
\begin{aligned}
  \frac{d}{dt}\frac{|x(\gamma(t))|^2}{2}-\left.\frac{d}{ds}\frac{|x(\gamma(s))|^2}{2}\right|_{s=0}&=\int_0^t\frac{d^2}{ds^2}\frac{|x(\gamma(s))|^2}{2}\,ds\\
  &=t-\frac{1}{\lambda}\int_0^t\inner{\HH}{A(\dot\gamma,\dot\gamma)}\,ds. 
\end{aligned}
\]
Since $|dx(\dot\gamma)|=1$, the Cauchy--Schwarz inequality gives
$$
  \left.\frac{d}{ds}\frac{|x(\gamma(s))|^2}{2}\right|_{s=0}
  =\inner{x(o)}{dx(\dot\gamma(0))}\ge-|x(o)|.
$$
For $1\le t\le\ell-1$, integrating \eqref{eq:properness-pointwise} along $\gamma$ and using \eqref{eq:properness-ricci-integral}, we obtain
$$
  \int_0^t\inner{\HH}{A(\dot\gamma,\dot\gamma)}\,ds\le\frac{1}{\mu}\int_0^t\operatorname{Ric}(\dot\gamma,\dot\gamma)\,ds\le\frac{C_0}{\mu}.
$$
Combining these bounds with the identity above, we obtain, for $1\le t\le\ell-1$,
$$
  \frac{d}{dt}\frac{|x(\gamma(t))|^2}{2}
  \ge-|x(o)|+t-\frac{C_0}{\lambda\mu}.
$$
Integrating once more from $1$ to $t$ and using $|x(\gamma(1))|^2\ge0$ gives, for $1\le t\le\ell-1$,
$$
  |x(\gamma(t))|^2\ge
  t^2-2\left(|x(o)|+\frac{C_0}{\lambda\mu}\right)t+2\left(|x(o)|+\frac{C_0}{\lambda\mu}\right)-1.
$$

Hence $|x(\gamma(t))|\ge t-C$ for all sufficiently large $t\le\ell-1$, where $C$ is independent of $q$ and $\gamma$.  Since $x$ is isometric, the Euclidean length of $x\circ\gamma$ on $[\ell-1,\ell]$ is one.  Thus, for all sufficiently large $\ell$,
$$
  |x(q)|\ge|x(\gamma(\ell-1))|-1\ge\ell-C-2.
$$
After enlarging $C$ to include the points in a fixed closed intrinsic ball, we obtain the estimate
$$
  d_M(o,q)\le|x(q)|+C,
  \qquad q\in M.
$$

Let $K\subset\R^{n+p}$ be compact and choose $R>0$ such that $K\subset\overline{B_R^{\R^{n+p}}(0)}$.  For $q\in x^{-1}(K)$, we have $|x(q)|\le R$ and hence $d_M(o,q)\le R+C$.  Thus $x^{-1}(K)\subset\overline{B_{R+C}^M(o)}$.  Since $x^{-1}(K)$ is closed in $M$ and $\overline{B_{R+C}^M(o)}$ is compact by Hopf--Rinow, $x^{-1}(K)$ is compact.  Therefore $x$ is proper. 
\end{proof}

Under the Ricci lower bounds, the following lemma shows that the mean curvature vector is nowhere vanishing unless the immersion is a linear subspace. 

\begin{lemma}\label{lem:mean curvature-dichotomy}
Let $x:M^n\to\R^{n+p}$ be a complete connected self-shrinker immersion. Suppose that, for some constant $0\le\mu<\frac{1}{n}$,
$$
  \operatorname{Ric}\ge\mu|\HH|^2g
$$
holds on $M$.  Then either $x(M)$ is an $n$-dimensional linear subspace through the origin or $|\HH|>0$ on $M$.
\end{lemma}

\begin{proof}
If $\HH\equiv0$, then $\operatorname{Ric}\ge0$, while the Gauss equation gives $R=-|A|^2\leq0$.  It follows that $A\equiv0$.  The self-shrinker equation gives $x^\perp\equiv0$. By completeness, $x(M)$ is an $n$-dimensional linear subspace through the origin.

We now consider the case $\HH\not\equiv0$.  Taking the trace of the Ricci lower bound and using the Gauss equation yields $|A|^2\le(1-n\mu)|\HH|^2$.  Differentiating the self-shrinker equation gives $\nabla_X^\perp\HH=\lambda A(X,x^T)$ for every tangent vector $X$. Thus
$$
  |\nabla_X^\perp\HH|\le\lambda\sqrt{1-n\mu}\,|x^T|\,|X|\,|\HH|.
$$
Suppose that $\HH(q)=0$ at some point $q\in M$.  For any $q'\in M$, choose a piecewise smooth curve $\gamma:[0,1]\to M$ joining $q$ to $q'$.  On each smooth segment, let $P_{s,t}:N_{\gamma(s)}M\to N_{\gamma(t)}M$ denote normal parallel transport.  Then
$$
  \HH(\gamma(t))=P_{0,t}\HH(q)
  +\int_0^tP_{s,t}\left(\nabla_{\dot\gamma(s)}^\perp\HH\right)\,ds.
$$

Since $P_{s,t}$ is an isometry, on each smooth segment there is a finite constant $C$ such that
$$
|\HH(\gamma(t))|\le C\int_0^t|\HH(\gamma(s))|\,ds.
$$  
Setting $F(t)=\int_0^t|\HH(\gamma(s))|\,ds$, we have $F(0)=0$, $F\ge0$, and $F'\le CF$. Therefore
$$
\bigl(e^{-Ct}F(t)\bigr)'\le0.
$$ 
Since $e^{-Ct}F(t)\ge0$ and its value at $t=0$ is zero, it follows that $F\equiv0$ and $\HH(q')=0$.   The arbitrariness of $q'$ implies $\HH\equiv0$, a contradiction.  Thus $|\HH|>0$ on $M$.
\end{proof}

The following lemma identifies a self-shrinker with parallel mean curvature vector as a compact minimal submanifold of the shrinker sphere.

\begin{lemma}\label{lem:spherical-reduction}
Let $x:M^n\to\R^{n+p}$ be a complete connected self-shrinker immersion with $|\HH|>0$.  Suppose that, for some constant $\mu>0$,
$$
  \operatorname{Ric}\ge\mu|\HH|^2g,
  \qquad
  \nabla^\perp\HH\equiv0.
$$
Then $M$ is compact and $x(M)$ is a minimal submanifold of $\mathbb S^{n+p-1}\left(\sqrt{\frac{n}{\lambda}}\right)$, with $|\HH|^2=n\lambda$.
\end{lemma}

\begin{proof}
Since $\nabla^\perp\HH\equiv0$, differentiating the self-shrinker equation gives $A(X,x^T)=0$ for every tangent vector $X$. The Gauss equation gives
$$
  \operatorname{Ric}(x^T,x^T)=\inner{\HH}{A(x^T,x^T)}
  -\sum_i|A(x^T,e_i)|^2=0.
$$
The Ricci lower bound now gives $0=\operatorname{Ric}(x^T,x^T)\ge\mu|\HH|^2|x^T|^2$.  Since $\mu>0$ and $|\HH|>0$, it follows that $x^T\equiv0$.

Since $\nabla|x|^2=2x^T=0$, $|x|$ is constant, and
$$
  0=\frac{1}{2}\Delta|x|^2=n+\inner{\HH}{x}=n-\lambda|x|^2.
$$
Thus $|x|^2=\frac{n}{\lambda}$.  Comparing $\HH=-\lambda x$ with the Euclidean mean curvature vector of a submanifold of this sphere shows that the immersion is minimal in the sphere, and $|\HH|^2=\lambda^2|x|^2=n\lambda$.  Finally, $\operatorname{Ric}\ge n\lambda\mu g>0$, so the Bonnet--Myers theorem implies that $M$ is compact.
\end{proof}

\section{Algebraic estimates under Ricci curvature pinching}\label{sec:ricci-pinching}

This section establishes the joint pointwise estimate needed for Theorem \ref{thm:ricci-improved} and determines the corresponding dimensional constant $\mu_n$. Then we develop the critical codimension-two algebraic estimate.

We continue to use the normal frame adapted to the mean curvature vector. The zeroth-order expression in the quotient identity \eqref{eq:quotient} is the left-hand side of \eqref{hd-eq:algebra}.  The Li--Li matrix inequality \cite{LiLi1992} gives
$$
  \sum_{\alpha,\beta\ge n+2}\left[
    \left(\sum_{i,j}A_{ij}^\alpha A_{ij}^\beta\right)^2
    +|[A^\alpha,A^\beta]|^2\right]\le\frac{3}{2}|A^\perp|^4.
$$
At a fixed point, choose a tangent frame in which $\mathring A^{n+1}$ is diagonal.  For every $\alpha\ge n+2$, the Cauchy--Schwarz inequality gives
$$
  \left(\sum_{i,j}\mathring A_{ij}^{n+1}A_{ij}^\alpha\right)^2
  =\left(\sum_i\mathring A_{ii}^{n+1}A_{ii}^\alpha\right)^2
  \le|\mathring A^{n+1}|^2\sum_i\left(A_{ii}^\alpha\right)^2.
$$
Moreover,
$$
  |[\mathring A^{n+1},A^\alpha]|^2
  =\sum_{i\ne j}\left(\mathring A_{ii}^{n+1}-\mathring A_{jj}^{n+1}\right)^2
  \left(A_{ij}^\alpha\right)^2
  \le\left(\max_i\mathring A_{ii}^{n+1}-\min_i\mathring A_{ii}^{n+1}\right)^2
  \sum_{i\ne j}\left(A_{ij}^\alpha\right)^2.
$$
Summing over $\alpha\ge n+2$ and substituting these estimates into \eqref{hd-eq:algebra} gives
\begin{equation}\label{eq:ricci-zero-order-first-bound}
\begin{aligned}
  &\sum_{\alpha,\beta}
  \left(\sum_{i,j}A_{ij}^\alpha A_{ij}^\beta\right)^2
  +\rho^\perp
  -
  \frac{|A|^2}{|\HH|^2}
  \sum_{\alpha,\beta}
  \left(\sum_{i,j}A_{ij}^\alpha A_{ij}^\beta\right)
  H^\alpha H^\beta                                      \\
  \le&|A^\perp|^2
  \Bigg[
    2\max\left\{
      |\mathring A^{n+1}|^2,
      \left(
        \max_i\mathring A_{ii}^{n+1}
        -
        \min_i\mathring A_{ii}^{n+1}
      \right)^2
    \right\}
    -
    |\mathring A^{n+1}|^2
    -
    \frac{|\HH|^2}{n}
    +
    \frac{3}{2}|A^\perp|^2
  \Bigg].
\end{aligned}
\end{equation}

\begin{proposition}\label{prop:ricci-zero-order}
Let $n\ge3$, and let $\mu$ satisfy
$$
  \frac{n-2}{n^2}<\mu<\frac{n-1}{n^2}.
$$
Assume $|\HH|>0$, and use the adapted normal frame introduced above at the point under consideration.  If
$$
  \operatorname{Ric}\ge\mu|\HH|^2g,
$$
then
$$
\begin{aligned}
  &\frac{|A|^2}{|\HH|^2}
  \sum_{\alpha,\beta}
  \left(\sum_{i,j}A_{ij}^\alpha A_{ij}^\beta\right)
  H^\alpha H^\beta
  -
  \sum_{\alpha,\beta}
  \left(\sum_{i,j}A_{ij}^\alpha A_{ij}^\beta\right)^2
  -\rho^\perp\\
  \ge&
  |A^\perp|^2|\HH|^2
  \Bigg\{
    \frac{1}{n}
    -
    \frac{3}{2}
    \left(
      \frac{n-1}{n}-n\mu
    \right)-
    \frac{15n}{2(n+3)}
    \left[
      \frac{1}{n}
      -
      \frac{1-\sqrt{1-4\mu}}{2}
    \right]^2
  \Bigg\}.
\end{aligned}
$$
\end{proposition}

\begin{proof}
Since each $A^\alpha$ is symmetric, $\left(\frac{A^\alpha}{|\HH|}\right)^2$ is positive semidefinite.  Hence, in the adapted normal frame, the Gauss equation and the Ricci lower bound give
$$
  \frac{A^{n+1}}{|\HH|}
  -\left(\frac{A^{n+1}}{|\HH|}\right)^2
  =\frac{\operatorname{Ric}}{|\HH|^2}
  +\sum_{\alpha\ge n+2}\left(\frac{A^\alpha}{|\HH|}\right)^2\ge\mu.
$$
Choose a tangent frame in which $A^{n+1}$ is diagonal. Applying this quadratic form inequality to the $i$-th basis vector gives
$$
  \left(\frac{A_{ii}^{n+1}}{|\HH|}-\frac{1}{2}\right)^2\le\frac{1}{4}-\mu.
$$
Since
$\mu<\frac{n-1}{n^2}\le\frac{1}{4}$, solving this quadratic inequality yields
$$
  \frac{1-\sqrt{1-4\mu}}{2} \le \frac{A_{ii}^{n+1}}{|\HH|} \le \frac{1+\sqrt{1-4\mu}}{2}.
$$
Moreover, $1-4\mu\ge\frac{(n-2)^2}{n^2}$, and hence $\frac{1-\sqrt{1-4\mu}}{2}\le\frac{1}{n}$.  Subtracting $\frac{1}{n}$ from the lower endpoint gives the lower bound below.  Since all diagonal entries satisfy this lower bound and their sum is zero, each is at most $n-1$ times the absolute value of that lower bound.  Therefore
\begin{align}
  -|\HH|
  \left(
    \frac{1}{n}
    -
    \frac{1-\sqrt{1-4\mu}}{2}
  \right)
  \le
  \mathring A_{ii}^{n+1}
  \le
  (n-1)|\HH|
  \left(
    \frac{1}{n}
    -
    \frac{1-\sqrt{1-4\mu}}{2}
  \right).
  \label{eq:ricci-spectral-box}
\end{align}

Taking the trace of the Ricci lower bound and using the Gauss equation, we obtain
$$
  |A^\perp|^2 \le \left( \frac{n-1}{n}-n\mu \right)|\HH|^2 - |\mathring A^{n+1}|^2.
$$
After substituting this bound into the $\frac{3}{2}|A^\perp|^2$ term in \eqref{eq:ricci-zero-order-first-bound}, it remains to prove the joint spectral estimate
\begin{equation}\label{eq:joint-spectral-estimate}
\begin{aligned}
  &2\max\left\{
    |\mathring A^{n+1}|^2,
    \left(
      \max_i\mathring A_{ii}^{n+1}
      -
      \min_i\mathring A_{ii}^{n+1}
    \right)^2
  \right\}
  -
  \frac{5}{2}|\mathring A^{n+1}|^2\\
  \le&
  \frac{15n}{2(n+3)}
  |\HH|^2
  \left[
    \frac{1}{n}
    -
    \frac{1-\sqrt{1-4\mu}}{2}
  \right]^2.
\end{aligned}
\end{equation}

We first consider the case
$$
  |\mathring A^{n+1}|^2\ge
  \left(\max_i\mathring A_{ii}^{n+1}-\min_i\mathring A_{ii}^{n+1}\right)^2.
$$
In this case, the left-hand side of \eqref{eq:joint-spectral-estimate} becomes
$$
  2|\mathring A^{n+1}|^2 - \frac{5}{2}|\mathring A^{n+1}|^2 = -\frac{1}{2}|\mathring A^{n+1}|^2 \le0.
$$
Since the right-hand side is nonnegative, \eqref{eq:joint-spectral-estimate} follows in this case.

We now consider the remaining case
$$
  |\mathring A^{n+1}|^2<
  \left(\max_i\mathring A_{ii}^{n+1}-\min_i\mathring A_{ii}^{n+1}\right)^2.
$$
Then $\mathring A^{n+1}\ne0$, and its trace-free property implies
$$
  \max_i\mathring A_{ii}^{n+1}>0, \qquad \min_i\mathring A_{ii}^{n+1}<0.
$$
Define the dimensionless ratio
$$
  t=\frac{\max_i\mathring A_{ii}^{n+1}}{-\min_i\mathring A_{ii}^{n+1}}.
$$
From the definition of $t$,
$$
\begin{aligned}
  \max_i\mathring A_{ii}^{n+1}
  &=
  t
  \left(
    -\min_i\mathring A_{ii}^{n+1}
  \right),\\
  \max_i\mathring A_{ii}^{n+1}
  -
  \min_i\mathring A_{ii}^{n+1}
  &=
  (t+1)
  \left(
    -\min_i\mathring A_{ii}^{n+1}
  \right),\\
  \max_i\mathring A_{ii}^{n+1}
  +
  \min_i\mathring A_{ii}^{n+1}
  &=
  (t-1)
  \left(
    -\min_i\mathring A_{ii}^{n+1}
  \right).
\end{aligned}
$$

After relabeling, assume
$$
  \mathring A_{11}^{n+1}=\max_i\mathring A_{ii}^{n+1},
  \qquad \mathring A_{22}^{n+1}=\min_i\mathring A_{ii}^{n+1}.
$$
From $\operatorname{tr}\mathring A^{n+1}=0$, we have
$$
  \sum_{i=3}^n\mathring A_{ii}^{n+1}
  =-\left(\max_i\mathring A_{ii}^{n+1}+\min_i\mathring A_{ii}^{n+1}\right).
$$
The Cauchy--Schwarz inequality therefore gives
$$
  \sum_{i=3}^n\left(\mathring A_{ii}^{n+1}\right)^2
  \ge\frac{1}{n-2}\left(\sum_{i=3}^n\mathring A_{ii}^{n+1}\right)^2
  =\frac{1}{n-2}\left(\max_i\mathring A_{ii}^{n+1}
    +\min_i\mathring A_{ii}^{n+1}\right)^2.
$$
Since $\mathring A^{n+1}$ is diagonal in the chosen frame,
$$
\begin{aligned}
  |\mathring A^{n+1}|^2
  &=
  \left(\max_i\mathring A_{ii}^{n+1}\right)^2
  +
  \left(\min_i\mathring A_{ii}^{n+1}\right)^2
  +
  \sum_{i=3}^n
  \left(\mathring A_{ii}^{n+1}\right)^2\\
  &\ge
  \left(\max_i\mathring A_{ii}^{n+1}\right)^2
  +
  \left(\min_i\mathring A_{ii}^{n+1}\right)^2
  +
  \frac{
    \left(
      \max_i\mathring A_{ii}^{n+1}
      +
      \min_i\mathring A_{ii}^{n+1}
    \right)^2
  }{n-2}\\
  &=
  \left(
    -\min_i\mathring A_{ii}^{n+1}
  \right)^2
  \left[
    t^2+1+\frac{(t-1)^2}{n-2}
  \right].
\end{aligned}
$$
A direct computation gives
$$
  2(t+1)^2-\frac{5}{2}\left(t^2+1+\frac{(t-1)^2}{n-2}\right)
  =\frac{15n}{2(n+3)}-\frac{n+3}{2(n-2)}
  \left(t-\frac{4n-3}{n+3}\right)^2\le\frac{15n}{2(n+3)}.
$$
Together with the preceding lower bound for $|\mathring A^{n+1}|^2$ and \eqref{eq:ricci-spectral-box}, this yields
$$
\begin{aligned}
  &2\left(
    \max_i\mathring A_{ii}^{n+1}
    -
    \min_i\mathring A_{ii}^{n+1}
  \right)^2
  -
  \frac{5}{2}|\mathring A^{n+1}|^2\\
  \le&
  \frac{15n}{2(n+3)}
  \left(-\min_i\mathring A_{ii}^{n+1}\right)^2\\
  \le&
  \frac{15n}{2(n+3)}|\HH|^2
  \left[
    \frac{1}{n}
    -
    \frac{1-\sqrt{1-4\mu}}{2}
  \right]^2.
\end{aligned}
$$
Therefore, \eqref{eq:joint-spectral-estimate} holds in both cases.

Combining the estimate for $|A^\perp|^2$, \eqref{eq:ricci-zero-order-first-bound} and \eqref{eq:joint-spectral-estimate} yields
$$
\begin{aligned}
  &\sum_{\alpha,\beta}
  \left(\sum_{i,j}A_{ij}^\alpha A_{ij}^\beta\right)^2
  +\rho^\perp
  -
  \frac{|A|^2}{|\HH|^2}
  \sum_{\alpha,\beta}
  \left(\sum_{i,j}A_{ij}^\alpha A_{ij}^\beta\right)
  H^\alpha H^\beta\\
  \le&
  |A^\perp|^2|\HH|^2
  \Bigg\{
      \frac{3}{2}
      \left(
      \frac{n-1}{n}-n\mu
    \right)
    +
    \frac{15n}{2(n+3)}
    \left[
      \frac{1}{n}
      -
      \frac{1-\sqrt{1-4\mu}}{2}
    \right]^2
    -
    \frac{1}{n}
  \Bigg\}.
\end{aligned}
$$
Reversing the last inequality yields precisely the asserted estimate.
\end{proof}

The coefficient on the right-hand side of the inequality above vanishes when
\begin{equation}\label{eq:mu-identity}
\begin{aligned}
  \frac{3}{2}
  \left(
    \frac{n-1}{n}-n\mu
  \right)
  +
  \frac{15n}{2(n+3)}
  \left[
    \frac{1}{n}
    -
    \frac{1-\sqrt{1-4\mu}}{2}
  \right]^2
  =
  \frac{1}{n}.
\end{aligned}
\end{equation}
The following proposition determines the corresponding value of $\mu$.

\begin{proposition}
\label{prop:mu-constant}
For each $n\ge3$, the pinching constant $\mu_n$ in Theorem \ref{thm:ricci-improved} is the unique solution of \eqref{eq:mu-identity} satisfying
$$
  \frac{n-2}{n^2}<\mu_n<\frac{n-1}{n^2}.
$$
Moreover,
$$
  \mu_n>\frac{3n-5}{3n^2}.
$$
If $\varepsilon_n=\mu_n-\frac{n-2}{n^2}$, then
$$
  \varepsilon_n=\frac{1}{3n^2}+\frac{20}{9n^5}
  +O\left(\frac{1}{n^6}\right)\qquad(n\to\infty).
$$
    
\end{proposition}

\begin{proof}
For $\frac{n-2}{n^2}<\mu<\frac{n-1}{n^2}$,  the upper bound for $\mu$ gives $n\sqrt{1-4\mu}-(n-2)>0$. Moreover,
$$
  \frac{1}{n}-\frac{1-\sqrt{1-4\mu}}{2}
  =
  \frac{n\sqrt{1-4\mu}-(n-2)}{2n}
$$
and
$$
\begin{aligned}
  \frac{n-1}{n}-n\mu
  &=
  \frac{n^2(1-4\mu)-(n-2)^2}{4n}\\
  &=
  \frac{
    \bigl(n\sqrt{1-4\mu}-(n-2)\bigr)
    \bigl(n\sqrt{1-4\mu}+(n-2)\bigr)
  }{4n}.
\end{aligned}
$$
Substituting these identities into \eqref{eq:mu-identity}, we obtain the quadratic equation
$$
  3(n+8)
  \bigl(n\sqrt{1-4\mu}-(n-2)\bigr)^2+
  6(n+3)(n-2)
  \bigl(n\sqrt{1-4\mu}-(n-2)\bigr)-8(n+3)=0.
$$
The leading coefficient of this quadratic polynomial is positive, whereas its constant term is negative.  Its two roots therefore have opposite signs, so exactly one of them is positive.  Hence
$$
  n\sqrt{1-4\mu}-(n-2)
  =
  \frac{
    \sqrt{
      (n+3)^2(n-2)^2
      +\frac{8}{3}(n+8)(n+3)
    }
    -(n+3)(n-2)
  }{n+8}.
$$
Since the map
$\mu\mapsto n\sqrt{1-4\mu}-(n-2)$ is strictly decreasing on the prescribed interval, \eqref{eq:mu-identity} has at most one admissible solution.  We define $\mu_n$ by
$$
  \sqrt{1-4\mu_n}
  =
  \frac{
    \sqrt{
      (n+3)^2(n-2)^2
      +\frac{8}{3}(n+8)(n+3)
    }
    +5(n-2)
  }{n(n+8)}.
$$
By construction, $\mu_n$ satisfies \eqref{eq:mu-identity} and
$n\sqrt{1-4\mu_n}-(n-2)>0$.  Hence
$\mu_n<\frac{n-1}{n^2}$.

Set $\varepsilon_n=\mu_n-\frac{n-2}{n^2}$.  Since
$$
  \frac{n-1}{n}-n\mu_n
  =
  \frac{1}{n}-n\varepsilon_n,
$$
it follows from \eqref{eq:mu-identity} that
$$
  \varepsilon_n=\frac{1}{3n^2}
  +\frac{5}{n+3}\left[\frac{1}{n}-\frac{1-\sqrt{1-4\mu_n}}{2}\right]^2.
$$
The second term is positive, and hence
$$
  \mu_n>\frac{n-2}{n^2}+\frac{1}{3n^2}=\frac{3n-5}{3n^2}.
$$
Thus $\mu_n$ is the unique solution in the prescribed interval.

The corresponding value of $\mu_n$ can be written as
\begin{equation}\label{eq:mu-epsilon-exact}
\begin{aligned}
  \mu_n
  &=
  \frac{n-2}{n^2}+\varepsilon_n,\\
  \varepsilon_n
  &=
  \frac{1}{3n^2}
  +
  \frac{80(n+3)}
  {9n^2
    \left[
      \sqrt{
        (n+3)^2(n-2)^2
        +\frac{8}{3}(n+8)(n+3)
      }
      +(n+3)(n-2)
    \right]^2}.
\end{aligned}
\end{equation}
As $n\to\infty$,
$$
  \sqrt{(n+3)^2(n-2)^2+\frac{8}{3}(n+8)(n+3)}
  +(n+3)(n-2)=2n^2+O(n).
$$
It follows from \eqref{eq:mu-epsilon-exact} that
$$
  \varepsilon_n=\frac{1}{3n^2}+\frac{20}{9n^5}
  +O\left(\frac{1}{n^6}\right).
$$
\end{proof}

We now consider the lower bound in Theorem \ref{thm:critical-main}.  In Euclidean codimension two, there is only one shape operator orthogonal to the mean curvature direction, and the resulting Ricci matrix inequality controls all zeroth-order terms in the quotient identity.

\begin{proposition}\label{prop:critical-codim-two-algebra}
Let $n\ge4$, and let $X,Y$ be trace-free real symmetric $n\times n$ matrices satisfying
$$
  \operatorname{Id}+(n-2)X-X^2-Y^2\ge0.
$$
Then
\begin{equation}\label{eq:critical-codim-two-algebra}
  2\inner{X}{Y}^2+2|[X,Y]|^2+|Y|^4
  -\left(n+|X|^2\right)|Y|^2\le0.
\end{equation}
\end{proposition}

\begin{proof}
The proof reduces the desired inequality to controlling a positive off-diagonal contribution after diagonalizing $X$. Introducing the positive semidefinite matrix $P$, we estimate this contribution using $\operatorname{tr}P$ and $[P,Y]$. The cases $n=4$ and $n\ge5$ are treated by pairwise and spectral estimates, respectively.

Choose an orthonormal frame in which
$X=\operatorname{diag}(x_1,\ldots,x_n)$, and set
$$
  P=\operatorname{Id}+(n-2)X-X^2-Y^2\ge0.
$$
Taking the trace gives
$$
  \operatorname{tr}P=n-|X|^2-|Y|^2.
$$
Thus the desired inequality is equivalent to
$$
  \inner{X}{Y}^2+|[X,Y]|^2
  \le\left(|X|^2+\frac{\operatorname{tr}P}{2}\right)|Y|^2.
$$
Since $X$ is diagonal, the Cauchy--Schwarz inequality and the commutator formula give
$$
\begin{aligned}
  \inner{X}{Y}^2
  &\le |X|^2\sum_iY_{ii}^2,\\
  |[X,Y]|^2
  &=\sum_{i\ne j}(x_i-x_j)^2Y_{ij}^2.
\end{aligned}
$$
Writing $[a]_+=\max\{a,0\}$, we consequently have
$$
  \inner{X}{Y}^2+|[X,Y]|^2
  \le|X|^2|Y|^2+\sum_{i\ne j}
  \left[(x_i-x_j)^2-|X|^2\right]_+Y_{ij}^2.
$$
Thus the desired inequality follows from the sufficient estimate
\begin{equation}\label{eq:off-diagonal-excess}
\sum_{i\ne j}\left[(x_i-x_j)^2-|X|^2\right]_+Y_{ij}^2
\le\frac{\operatorname{tr}P}{2}|Y|^2.
\end{equation}

We first consider the case $n=4$.  Since
$$
  P+Y^2=\operatorname{Id}+2X-X^2\ge0,
$$
taking the $i$-th diagonal entry gives $1+2x_i-x_i^2\ge0$, and hence $x_i\ge1-\sqrt2>-\frac{1}{2}$. For distinct indices $i,j$, let $k,l$ denote the remaining two indices.  The trace-free condition gives
$$
  x_i+x_j=-(x_k+x_l)\le1.
$$
The definition of $P$ and the identity $[Y^2,Y]=0$ give
$$
  [P,Y]_{ij}=\bigl[(2x_i-x_i^2)-(2x_j-x_j^2)\bigr]Y_{ij}=(x_i-x_j)(2-x_i-x_j)Y_{ij}.
$$
For every $i\ne j$, the key pairwise estimate is
\begin{equation}\label{eq:four-dimensional-pair-estimate}
  (x_i-x_j)^2(2-x_i-x_j)^2
  \ge2(4-|X|^2)\left[(x_i-x_j)^2-|X|^2\right]_+.
\end{equation}

Since $x_k+x_l=-(x_i+x_j)$, the Cauchy--Schwarz inequality gives
$$
  |X|^2=x_i^2+x_j^2+x_k^2+x_l^2\ge
  (x_i+x_j)^2+\frac{(x_i-x_j)^2}{2}.
$$
Moreover, $\operatorname{tr}P\ge0$ gives $|X|^2\le4$. If $(x_i-x_j)^2\le|X|^2$, the claim is immediate. We may therefore assume that $(x_i-x_j)^2>|X|^2$, which also implies $(x_i-x_j)^2>2(x_i+x_j)^2$.

For fixed $x_i+x_j$ and $x_i-x_j$, the function
$$
  r\longmapsto (x_i-x_j)^2(2-x_i-x_j)^2
  -2(4-r)\bigl((x_i-x_j)^2-r\bigr)
$$
is strictly increasing whenever $r<(x_i-x_j)^2$ and $r\le4$, since its derivative is $2((x_i-x_j)^2+4-2r)>0$. Using the preceding lower bound for $|X|^2$, we obtain
$$
\begin{aligned}
  &(x_i-x_j)^2(2-x_i-x_j)^2
  -2(4-|X|^2)\bigl((x_i-x_j)^2-|X|^2\bigr)\\
  \ge&\frac{(x_i-x_j)^4}{2}
  +\bigl((x_i+x_j)^2-4(x_i+x_j)\bigr)(x_i-x_j)^2+8(x_i+x_j)^2-2(x_i+x_j)^4.
\end{aligned}
$$
If $x_i+x_j\le0$, the last expression is increasing as a function of $(x_i-x_j)^2$ on $(x_i-x_j)^2\ge2(x_i+x_j)^2$, and hence it is bounded below by $2(x_i+x_j)^2(x_i+x_j-2)^2$. If $0<x_i+x_j\le1$, completing the square gives
$$
\begin{aligned}
  &\frac{(x_i-x_j)^4}{2}
  +\bigl((x_i+x_j)^2-4(x_i+x_j)\bigr)(x_i-x_j)^2
  +8(x_i+x_j)^2-2(x_i+x_j)^4\\
  =&\frac{1}{2}
  \left[(x_i-x_j)^2-(x_i+x_j)(4-x_i-x_j)\right]^2+\frac{(x_i+x_j)^3}{2}\bigl(8-5(x_i+x_j)\bigr)
  \ge0.
\end{aligned}
$$
Consequently, multiplying \eqref{eq:four-dimensional-pair-estimate} by $Y_{ij}^2$ and summing over $i\ne j$, we obtain
$$
  |[P,Y]|^2\ge2(4-|X|^2)
  \sum_{i\ne j}
  \left[(x_i-x_j)^2-|X|^2\right]_+Y_{ij}^2.
$$
If $\operatorname{tr}P>0$, diagonalizing $P$ and using its nonnegative eigenvalues gives
$$
  |[P,Y]|^2\le(\operatorname{tr}P)^2|Y|^2.
$$
Since $\operatorname{tr}P\le4-|X|^2$, the preceding two estimates give
$$
  \sum_{i\ne j}
  \left[(x_i-x_j)^2-|X|^2\right]_+Y_{ij}^2
  \le\frac{(\operatorname{tr}P)^2}{2(4-|X|^2)}|Y|^2
  \le\frac{\operatorname{tr}P}{2}|Y|^2.
$$
This estimate also holds when $\operatorname{tr}P=0$.  Indeed, then $P=0$; the preceding lower bound makes the sum vanish if $4-|X|^2>0$, whereas if $4-|X|^2=0$, the trace identity gives $Y=0$.  Thus the sufficient estimate above holds for $n=4$, and the desired inequality follows.

We next consider the case $n\ge5$.  Since
$$
  P+Y^2=\operatorname{Id}+(n-2)X-X^2\ge0,
$$
we have $1+(n-2)x_i-x_i^2\ge0$, and hence
$$
  x_i\ge\frac{n-2-\sqrt{(n-2)^2+4}}{2}>-\frac{1}{n-2}.
$$
Using $\sum_i x_i=0$, for $i\ne j$ we obtain
$$
  x_i+x_j=-\sum_{k\ne i,j}x_k\le1,
$$
so that $n-2-x_i-x_j\ge n-3$.  The trace-free condition also gives
$$
  0\le\max_i x_i\le-(n-1)\min_i x_i\le\frac{n-1}{n-2}.
$$
Since $|X|^2\ge(\max_i x_i)^2+(\min_i x_i)^2$, these bounds imply
$$
  \left[\left(\max_i x_i-\min_i x_i\right)^2-|X|^2\right]_+
  \le2\left(\max_i x_i\right)\left(-\min_i x_i\right)
  \le\frac{2(n-1)}{(n-2)^2}.
$$
For $n\ge5$, the elementary inequality $\frac{8(n-1)}{(n-2)^2}\le(n-3)^2$ therefore yields
$$
  4\left[\left(\max_i x_i-\min_i x_i\right)^2-|X|^2\right]_+
  \le(n-3)^2.
$$
Suppose first that
$$
  \operatorname{tr}P\ge
  2\left[\left(\max_i x_i-\min_i x_i\right)^2-|X|^2\right]_+.
$$
Then
$$
  \sum_{i\ne j}
  \left[(x_i-x_j)^2-|X|^2\right]_+Y_{ij}^2\le
  \left[\left(\max_i x_i-\min_i x_i\right)^2-|X|^2\right]_+|Y|^2
  \le\frac{\operatorname{tr}P}{2}|Y|^2.
$$
The required inequality now follows.

We next consider the complementary case
$$
  \operatorname{tr}P<
  2\left[\left(\max_i x_i-\min_i x_i\right)^2-|X|^2\right]_+.
$$
The preceding spectral estimate gives $\operatorname{tr}P<\frac{(n-3)^2}{2}$.  Moreover,
$$
  [P,Y]_{ij}
  =(x_i-x_j)(n-2-x_i-x_j)Y_{ij},
$$
and hence
$$
  (n-3)^2|[X,Y]|^2\le|[P,Y]|^2.
$$
Choose a frame in which $P$ is diagonal, with eigenvalues $p_i\ge0$. Then
$$
  |[P,Y]|^2=\sum_{i,j}(p_i-p_j)^2Y_{ij}^2
  \le(\operatorname{tr}P)^2|Y|^2,
$$
and therefore
$$
  |[X,Y]|^2\le\frac{(\operatorname{tr}P)^2}{(n-3)^2}|Y|^2
  \le\frac{\operatorname{tr}P}{2}|Y|^2.
$$
Together with $\inner{X}{Y}^2\le|X|^2|Y|^2$, this proves the required inequality in the complementary case and completes the proof.
\end{proof}

\section{Proofs of the main theorems}\label{sec:main-proofs}

We now combine the auxiliary geometric results of Section~2 with the algebraic estimates of Section~3.  The form of the zeroth-order estimate depends on the dimension: for $n=2$, the identities $\operatorname{Ric}=Kg$ and the algebra of trace-free symmetric $2\times2$ matrices give direct control, while for $n\ge3$ the required estimate is supplied by Proposition \ref{prop:ricci-zero-order}.

\begin{proof}[Proof of Theorem \ref{thm:ricci-improved}]
For $n=2$, $\mu_2=\frac{1}{6}<\frac{1}{2}$, while for $n\ge3$, Proposition \ref{prop:mu-constant} gives $0<\mu_n<\frac{1}{n}$. Lemma \ref{lem:mean curvature-dichotomy} therefore shows that either $x(M)$ is a linear subspace through the origin, in which case the conclusion follows, or $|\HH|>0$ on $M$.  We henceforth consider the latter case.  Since $\mu_n>0$, Lemma \ref{lem:ricci-properness} shows that the immersion is proper.  By Cheng--Zhou \cite[Theorem 4.1]{ChengZhou2013}, it has polynomial volume growth, so the weighted Liouville lemma applies below.

We first consider $n=2$.  Since $\operatorname{Ric}=Kg$, the Ricci lower bound and the Gauss equation yield
$$
  0\le\frac{|\mathring A|^2}{|\HH|^2}
  =\frac{1}{2}-\frac{2K}{|\HH|^2}\le\frac{1}{6}.
$$
For trace-free symmetric $2\times2$ matrices $B,C$, the identity
$$
  \inner{B}{C}^2+\frac{1}{2}|[B,C]|^2=|B|^2|C|^2
$$
reduces the quotient identity \eqref{eq:quotient} to
$$
  \frac{1}{2}\left(\mathcal L_\lambda
    +2\inner{\nabla\log|\HH|^2}{\nabla\cdot}\right)
  \frac{|\mathring A|^2}{|\HH|^2}
  =\frac{1}{|\HH|^2}\left(|\nabla\mathring A|^2
    -\frac{|\mathring A|^2}{|\HH|^2}|\nabla^\perp\HH|^2
    +2K|A^\perp|^2-\frac{1}{2}\rho^\perp\right).
$$
In the adapted normal frame $e_3=\HH/|\HH|$, the B\"ottcher--Wenzel and DDVV inequalities \cite{BottcherWenzel2005,GeTang2008,Lu2011} give
$$
  \rho^\perp\le\left(4|\mathring A^3|^2+|A^\perp|^2\right)|A^\perp|^2.
$$
Combined with the Gauss equation and the quotient bound, this yields
$$
  2K|A^\perp|^2-\frac{1}{2}\rho^\perp
  \ge\frac{3}{2}|A^\perp|^4.
$$

Lemma \ref{lem:gradient-estimate} and the quotient bound show that the gradient contribution is nonnegative.  Hence
$$
  \frac{1}{2}\left(\mathcal L_\lambda
    +2\inner{\nabla\log|\HH|^2}{\nabla\cdot}\right)
  \frac{|\mathring A|^2}{|\HH|^2}\ge\frac{3|A^\perp|^4}{2|\HH|^2}\ge0.
$$
The quotient is bounded, so Lemma \ref{lem:weighted-liouville} shows that it is constant.  The left-hand side therefore vanishes, and hence $A^\perp\equiv0$. Substituting into the exact quotient identity and applying Lemma \ref{lem:gradient-estimate} together with the quotient bound gives
$$
  0=|\nabla\mathring A|^2
  -\frac{|\mathring A|^2}{|\HH|^2}|\nabla^\perp\HH|^2
  \ge\left(\frac{1}{4}-\frac{|\mathring A|^2}{|\HH|^2}\right)|\nabla^\perp\HH|^2
  \ge\frac{1}{12}|\nabla^\perp\HH|^2\ge0.
$$
Hence $\nabla^\perp\HH\equiv0$.  Lemma \ref{lem:spherical-reduction} now shows that $M$ is minimal in $\mathbb S^{p+1}\left(\sqrt{\frac{2}{\lambda}}\right)$. Under this reduction, $A^\perp$ is the second fundamental form in the shrinker sphere, so $A^\perp\equiv0$ makes the spherical immersion totally geodesic.  By completeness, it is a covering of $\mathbb S^2\left(\sqrt{\frac{2}{\lambda}}\right)$; since the target sphere is simply connected, the covering is one-sheeted.  This proves the theorem for $n=2$.

We now assume $n\ge3$. For $n=3$, Proposition \ref{prop:mu-constant} yields $\mu_3>\frac{4}{27}>\frac{2}{15}$, so
$$
  \frac{|\mathring A|^2}{|\HH|^2} < \frac{4}{15} = \frac{2(n-1)}{n(n+2)}.
$$
For $n\ge4$, the Ricci lower bound and the Gauss equation, together with the same proposition, give
$$
  \frac{|\mathring A|^2}{|\HH|^2} < \frac{1}{n} \le \frac{2(n-1)}{n(n+2)}.
$$
Lemma \ref{lem:gradient-estimate} now yields
$$
  |\nabla\mathring A|^2
  -\frac{|\mathring A|^2}{|\HH|^2}|\nabla^\perp\HH|^2
  \ge\left(\frac{2(n-1)}{n(n+2)}-\frac{|\mathring A|^2}{|\HH|^2}\right)
  |\nabla^\perp\HH|^2\ge0,
$$
where the coefficient in parentheses is strictly positive.  By Proposition \ref{prop:mu-constant}, $\mu_n$ lies in the interval required by Proposition \ref{prop:ricci-zero-order} and satisfies \eqref{eq:mu-identity}. Combining that proposition and Lemma \ref{lem:gradient-estimate} with the quotient identity \eqref{eq:quotient}, we obtain
$$
  \frac{1}{2}\left(\mathcal L_\lambda
    +2\inner{\nabla\log|\HH|^2}{\nabla\cdot}\right)
  \frac{|\mathring A|^2}{|\HH|^2}
  \ge\left(\frac{2(n-1)}{n(n+2)}
    -\frac{|\mathring A|^2}{|\HH|^2}\right)
  \frac{|\nabla^\perp\HH|^2}{|\HH|^2}\ge0.
$$
The quotient is bounded by the trace estimate, so Lemma \ref{lem:weighted-liouville} shows that it is constant.  The left-hand side therefore vanishes, and the strict positivity of the coefficient gives $\nabla^\perp\HH\equiv0$.

Lemma \ref{lem:spherical-reduction} shows that $M$ is compact and minimal in $\mathbb S^{n+p-1}\left(\sqrt{\frac{n}{\lambda}}\right)$, with $|\HH|^2=n\lambda$.  After rescaling the ambient sphere to the unit sphere, the Ricci lower bound becomes
$$
  \operatorname{Ric} \ge n^2\mu_n g > (n-2)g.
$$
If $M$ is nonorientable, let $\pi:\widetilde M\to M$ be its orientable double cover and $\widetilde x=x\circ\pi$ the lifted immersion.  By Xu--Gu \cite[Corollary 3.4]{XuGu2013}, an oriented compact minimal submanifold $N^n\subset\mathbb S^{n+q}(1)$, $n\ge3$, satisfying $\operatorname{Ric}_N>(n-2)g$ is totally geodesic. It follows that $\widetilde x(\widetilde M)$ is a totally geodesic unit $n$-sphere.

The map $\widetilde x:\widetilde M\to\mathbb S^n(1)$ is a local isometry and, since $\widetilde M$ is compact, a covering map. As $n\ge3$, the target sphere is simply connected; hence this covering is one-sheeted.  In particular, the deck transformation of a nontrivial orientable double cover would fix $\widetilde x$, contradicting its injectivity.  Thus the double cover is trivial.  Scaling back, $x(M)$ is $\mathbb S^n\left(\sqrt{\frac{n}{\lambda}}\right)$ in an $(n+1)$-dimensional linear subspace of $\R^{n+p}$.
\end{proof}

At the critical coefficient, constancy of the quotient does not force $\nabla^\perp\HH=0$ when $n=4$. The following equality lemma instead deduces $x^T=0$ from the vanishing of the gradient term, yielding the spherical reduction.
\begin{lemma}\label{lem:critical-gradient-equality}
Let $x:M^n\to\R^{n+p}$, $n\ge2$, be a self-shrinker immersion with $|\HH|>0$ and positive Ricci curvature.  Suppose that $\frac{|\mathring A|^2}{|\HH|^2}$ is constant, satisfying
$$
  \frac{|\mathring A|^2}{|\HH|^2}\le\frac{2(n-1)}{n(n+2)}
\mbox{\,\,\,\,\,\,and\,\,\,\,\,\,}
  |\nabla\mathring A|^2
  -\frac{|\mathring A|^2}{|\HH|^2}|\nabla^\perp\HH|^2=0.
$$
Then $x^T\equiv0$.
\end{lemma}

\begin{proof}
By Lemma \ref{lem:gradient-estimate},
$$
  0=|\nabla\mathring A|^2
  -\frac{|\mathring A|^2}{|\HH|^2}|\nabla^\perp\HH|^2
  \ge\left(\frac{2(n-1)}{n(n+2)}-\frac{|\mathring A|^2}{|\HH|^2}\right)
  |\nabla^\perp\HH|^2\ge0.
$$
Since the quotient is constant, suppose first that
$$
  \frac{|\mathring A|^2}{|\HH|^2}<\frac{2(n-1)}{n(n+2)}.
$$
The strict positivity of the coefficient above forces $\nabla^\perp\HH\equiv0$. Differentiating the self-shrinker equation gives $A(X,x^T)=0$ for every tangent vector $X$.  The Gauss equation then yields $\operatorname{Ric}(x^T,\cdot)=0$, and the positivity of the Ricci curvature implies $x^T\equiv0$.

It remains to consider
$$
  \frac{|\mathring A|^2}{|\HH|^2}=\frac{2(n-1)}{n(n+2)}.
$$
The assumed vanishing of the gradient contribution then gives equality in Lemma \ref{lem:gradient-estimate}.  Fix a point and choose local orthonormal tangent and normal frames that are geodesic and parallel at that point, respectively.  The orthogonal decomposition used in the gradient estimate is
$$
  |\nabla\mathring A|^2
  -\frac{2(n-1)}{n(n+2)}|\nabla^\perp\HH|^2=
  \sum_{i,j,k}
  \left|
    \nabla_k^\perp A_{ij}
    -\frac{1}{n+2}
    \left(
      \delta_{jk}\nabla_i^\perp\HH
      +\delta_{ik}\nabla_j^\perp\HH
      +\delta_{ij}\nabla_k^\perp\HH
    \right)
  \right|^2.
$$
Hence equality in the gradient estimate implies
$$
  \nabla_k^\perp A_{ij}=\frac{1}{n+2}\left(
    \delta_{jk}\nabla_i^\perp\HH+\delta_{ik}\nabla_j^\perp\HH
    +\delta_{ij}\nabla_k^\perp\HH\right).
$$
Contracting with $\mathring A$ and using its trace-free property gives
$$
  \inner{\nabla_k\mathring A}{\mathring A}
  =\frac{2}{n+2}\sum_i\inner{\nabla_i^\perp\HH}{\mathring A_{ik}}.
$$
On the other hand, differentiating the constant quotient gives
$$
  \inner{\nabla_k\mathring A}{\mathring A}
  =\frac{2(n-1)}{n(n+2)}\inner{\nabla_k^\perp\HH}{\HH}.
$$
Comparing these identities gives
$$
  \sum_i\inner{\nabla_i^\perp\HH}{\mathring A_{ik}}
  =\frac{n-1}{n}\inner{\nabla_k^\perp\HH}{\HH}.
$$
Using $\mathring A_{ik}=A_{ik}-\frac{1}{n}\delta_{ik}\HH$ and then differentiating the self-shrinker equation, we obtain
$$
\begin{aligned}
  0
  &=
  \sum_i
  \inner{\nabla_i^\perp\HH}{\mathring A_{ik}}
  -\frac{n-1}{n}
  \inner{\nabla_k^\perp\HH}{\HH}\\
  &=
  \sum_i
  \inner{\nabla_i^\perp\HH}{A_{ik}}
  -\inner{\nabla_k^\perp\HH}{\HH}\\
  &=
  \lambda
  \left(
    \sum_i\inner{A(e_i,x^T)}{A_{ik}}
    -\inner{A(e_k,x^T)}{\HH}
  \right)\\
  &=-\lambda\operatorname{Ric}(x^T,e_k),
\end{aligned}
$$
where the last equality follows from the Gauss equation.  Thus $\operatorname{Ric}(x^T,\cdot)=0$, and hence $x^T\equiv0$.
\end{proof}

\begin{proof}[Proof of Theorem \ref{thm:critical-main}]
Since $0<\frac{n-2}{n^2}<\frac{1}{n}$, Lemma \ref{lem:ricci-properness} gives properness.  By Cheng--Zhou \cite[Theorem 4.1]{ChengZhou2013}, the immersion has polynomial volume growth, so the volume growth hypothesis in Lemma \ref{lem:weighted-liouville} is satisfied. Lemma \ref{lem:mean curvature-dichotomy} shows that either $x(M)$ is a linear subspace through the origin, in which case the conclusion follows, or $|\HH|>0$ on $M$.  We thus assume $|\HH|>0$ on $M$.  Tracing the Ricci lower bound and applying the Gauss equation with the trace-free decomposition yields $\frac{|\mathring A|^2}{|\HH|^2}\le\frac{1}{n}$.

In the case of $p=1$, we have $A^\perp\equiv0$, and by \eqref{hd-eq:algebra}, the zeroth-order expression
subtracted in the quotient identity \eqref{eq:quotient} vanishes.

In the case of $p=2$, choose the normal frame adapted to the mean curvature vector.  Since there is only one normal direction orthogonal to $\HH$, the only remaining shape operator is $A^{n+2}$.  The Gauss equation shows that the Ricci lower bound is equivalent to
$$
  \operatorname{Id}+\frac{n(n-2)}{|\HH|}\mathring A^{n+1}
  -\frac{n^2}{|\HH|^2}\left[(\mathring A^{n+1})^2+(A^{n+2})^2\right]\ge0.
$$
Applying Proposition \ref{prop:critical-codim-two-algebra} to $\frac{n\mathring A^{n+1}}{|\HH|}$ and $\frac{nA^{n+2}}{|\HH|}$, and expanding, we obtain
$$
\begin{aligned}
  &\sum_{\alpha,\beta}
  \left(\sum_{i,j}A_{ij}^\alpha A_{ij}^\beta\right)^2
  +\rho^\perp
  -
  \frac{|A|^2}{|\HH|^2}
  \sum_{\alpha,\beta}
  \left(\sum_{i,j}A_{ij}^\alpha A_{ij}^\beta\right)
  H^\alpha H^\beta\\
  =&
  2\inner{\mathring A^{n+1}}{A^{n+2}}^2
  +2|[\mathring A^{n+1},A^{n+2}]|^2
  +|A^{n+2}|^4-
  \left(
    |\mathring A^{n+1}|^2+\frac{|\HH|^2}{n}
  \right)|A^{n+2}|^2
  \le0.
\end{aligned}
$$

In the case of $p\ge3$, Lemma \ref{hd-lem:normal-scalar-algebra} and the Gauss equation give
$$
\begin{aligned}
  &\sum_{\alpha,\beta}
  \left(\sum_{i,j}A_{ij}^\alpha A_{ij}^\beta\right)^2
  +\rho^\perp
  -
  \frac{|A|^2}{|\HH|^2}
  \sum_{\alpha,\beta}
  \left(\sum_{i,j}A_{ij}^\alpha A_{ij}^\beta\right)
  H^\alpha H^\beta\\
  \le&
  \rho^\perp
  -
  \left(
    R-\frac{n-2}{n}|\HH|^2
  \right)|A^\perp|^2
  \le0.
\end{aligned}
$$
The last inequality is precisely the normal curvature assumption in the case $p\ge3$.  Moreover, the trace estimate above gives
$
  R-\frac{n-2}{n}|\HH|^2
  =\frac{1}{n}|\HH|^2-|\mathring A|^2\ge0
$, so the coefficient of $|A^\perp|^2$ is nonnegative.

In all cases, the displayed zeroth-order expression is nonpositive and hence contributes a nonnegative term to the right-hand side of \eqref{eq:quotient}.  Combining this observation with Lemma \ref{lem:gradient-estimate} and the bound $\frac{|\mathring A|^2}{|\HH|^2}\le\frac{1}{n}$, we obtain
$$
\begin{aligned}
  \frac{1}{2}
  \left(
    \mathcal L_\lambda
    +
    2\inner{\nabla\log|\HH|^2}{\nabla\cdot}
  \right)
  \frac{|\mathring A|^2}{|\HH|^2}\ge\left(
    \frac{2(n-1)}{n(n+2)}
    -
    \frac{|\mathring A|^2}{|\HH|^2}
  \right)\frac{|\nabla^\perp\HH|^2}{|\HH|^2}\ge
  \frac{n-4}{n(n+2)}
  \frac{|\nabla^\perp\HH|^2}{|\HH|^2}
  \ge0.
\end{aligned}
$$
The quotient is bounded, so Lemma \ref{lem:weighted-liouville} shows that it is constant.  To include the possible equality case when $n=4$, we return to the exact quotient identity \eqref{eq:quotient}.  Its left-hand side now vanishes, while its zeroth-order contribution is nonnegative; therefore the gradient contribution satisfies
$$
  |\nabla\mathring A|^2
  -\frac{|\mathring A|^2}{|\HH|^2}|\nabla^\perp\HH|^2\le0.
$$
On the other hand, Lemma \ref{lem:gradient-estimate} and $\frac{|\mathring A|^2}{|\HH|^2}\le\frac{1}{n}\le\frac{2(n-1)}{n(n+2)}$ show that this gradient contribution is nonnegative.  It must therefore vanish identically.  The quotient is constant and the Ricci curvature is positive, so all the hypotheses of Lemma \ref{lem:critical-gradient-equality} are satisfied.  Consequently, $x^T\equiv0$, so $|x|$ is constant and $x=x^\perp$.  The identity
$$
  \frac{1}{2}\Delta|x|^2=n+\inner{x}{\HH}=n-\lambda|x|^2
$$
then gives $|x|^2=\frac{n}{\lambda}$.  Thus $M$ is minimal in $\mathbb S^{n+p-1}\left(\sqrt{\frac{n}{\lambda}}\right)$ and $|\HH|^2=n\lambda$.  Since $x$ is proper and its image is contained in this compact sphere, $M$ is compact.  After rescaling the ambient sphere to the unit sphere, the Ricci lower bound becomes
$$
  \operatorname{Ric}\ge(n-2)g.
$$
Since $n\ge4$, the lower bound is strictly positive.  Lemma 2.3 of Xu--Gu \cite{XuGu2013} gives $\pi_1(M)=0$, so $M$ is orientable and Ejiri's classification theorem \cite{Ejiri1979} applies.  The alternatives are a totally geodesic sphere, a Clifford hypersurface, and the Veronese embedding of $\mathbb{CP}^2$.  In the Clifford case, $M$ is a minimal hypersurface in the unit sphere of the form $\mathbb S^m\left(\sqrt{\frac{m}{n}}\right)
  \times
  \mathbb S^{n-m}\left(\sqrt{\frac{n-m}{n}}\right)$ for some $1\le m\le n-1$.  Since
$\operatorname{Ric}_{\mathbb S^k(r)}=\frac{k-1}{r^2}g$, the Ricci curvatures along the two factors are $\frac{n(m-1)}{m}$ and $\frac{n(n-m-1)}{n-m}$. The condition $\operatorname{Ric}\ge(n-2)g$ yields $m\ge\frac n2$ and $n-m\ge\frac n2$, and hence $m=\frac n2$.  Thus $n$ is even, and the Clifford hypersurface is $\mathbb S^{\frac n2}\left(\frac1{\sqrt2}\right)\times\mathbb S^{\frac n2}\left(\frac1{\sqrt2}\right)$. The Veronese model has real dimension four and can therefore occur only when $n=4$. In this case, a smaller ambient sphere excludes it, while otherwise the additional normal curvature hypothesis applies. Indeed, on the Veronese model one has $R-\frac{n-2}{n}|\HH|^2=0$, so the hypothesis forces $\rho^\perp\equiv0$, contradicting the nonflat normal bundle.

Rescaling back yields the two submanifolds stated in the theorem.  The round sphere is simply connected; in the Clifford case, $n\ge4$ and both factors have dimension $\frac n2\ge2$, so this model is simply connected as well.  Hence no nontrivial covering occurs.
\end{proof}

\section*{Acknowledgments}
The authors would like to thank Professors Hongwei Xu and Entao Zhao for their encouragement and stimulating discussions.

\end{document}